\documentclass[11pt,twoside]{article}
\usepackage{mathrsfs}
\usepackage{amssymb}
\usepackage{amsmath}
\usepackage[all]{xy}
\usepackage{amsthm}
\numberwithin{equation}{section}

\newtheorem{theorem}{Theorem}[section]

\newtheorem{definition}[theorem]{Definition}
\newtheorem{remark}[theorem]{Remark}
\newtheorem{lemma}[theorem]{Lemma}
\newtheorem{example}[theorem]{Example}

\newtheorem{corollary}[theorem]{Corollary}

\newcommand{\edge}{\ar@{-}}

\newcommand{\pf}{\noindent\begin {proof}}
\newcommand{\epf}{\end{proof}}

\newcommand{\Hom}{\mbox{\rm Hom}}

\def\La{\mathop{\rm \Lambda}\nolimits}

\def\mod{\mathop{\rm mod}\nolimits}

\def\pd{\mathop{\rm pd}\nolimits}

\def\max{\mathop{\rm max}\nolimits}

\def\sup{\mathop{\rm sup}\nolimits}
\def\inf{\mathop{\rm inf}\nolimits}
\def\add{\mathop{\rm add}\nolimits}

\def\gldim{\mathop{\rm gl.dim}\nolimits}

\def\lFindim{\mathop{\rm l.Fin.dim}\nolimits}

\def\rad{\mathop{{\rm rad}}\nolimits}

\def\top{\mathop{{\rm top}}\nolimits}

\def\dim{\mathop{\rm dim}\nolimits}

\def\Hom{\mathop{\rm Hom}\nolimits}

\def\sup{\mathop{\rm sup}\nolimits}
\def\lim{\mathop{\underrightarrow{\rm lim}}\nolimits}

\def\gen{\mathop{\rm gen}\nolimits}

\def\End{\mathop{\rm End}\nolimits}

\def\C{\mathop{\rm \mathcal{C}}\nolimits}

\def\S{\mathop{\rm \mathcal{S}}\nolimits}

\def\V{\mathop{\rm \mathcal{V}}\nolimits}

\def\LL{\mathop{\rm LL}\nolimits}

\def\gen.dim{\mathop{\rm gen.dim}\nolimits}

\title{ \bf Radical layer length and syzygy-finite algebras 
  \thanks{2010 Mathematics Subject Classification: 18G25, 16E30, 16E10.}
\thanks{Keywords: derived categories, left big finitistic dimension,
radical layer length, syzygy-finite}}
\vspace{0.2cm}

\author { \ Junling  Zheng \thanks{Email: zhengjunling@cjlu.edu.cn}
\\
{\it \scriptsize  Department of Mathematics, China Jiliang University, Hangzhou, 310018, P. R. China }}

\date{ }

\begin{document}

\baselineskip=16pt

%%%%%%%%%%%%%%%%%%%%%%%%%%%%%%%%%%%%%%%%%

\maketitle

\begin{abstract}
Let $\Lambda$ be an artin algebra.
We obtain that $\Lambda$ is syzygy-finite when the radical
 layer length of $\Lambda$ is at most two; as two
consequences, we give a new upper bound for the dimension
 of the bounded derived category of
the category $\mod \Lambda$ of finitely generated right
 $\Lambda$-modules in terms of the projective
of certain class of simple right $\Lambda$-modules and also
 get the left big finitistic dimension conjecture holds.
\end{abstract}

\pagestyle{myheadings}
\markboth{\rightline {\scriptsize  }}
         {\leftline{\scriptsize  Radical layer length and syzygy-finite algebras }}

\tableofcontents    % Ŀ¼

\section{Introduction} %delete * to number this section
Given a triangulated category $\mathcal{T}$, Rouquier introduced in
 \cite{rouquier2006representation} the dimension $\dim\mathcal{T}$ of $\mathcal{T}$
under the idea of Bondal and van den Bergh in \cite{bondal2003generators}.
 This dimension and the infimum of the Orlov spectrum of $\mathcal{T}$
coincide, see \cite{orlov2009remarks,ballard2012orlov}. Roughly speaking,
 it is an invariant that measures how quickly 
 the category can be built from one object.
Many authors have studied the upper bound of $\dim \mathcal{T}$,
 see \cite{ballard2012orlov, bergh2015gorenstein, chen2008algebras,
  han2009derived, krause2006rouquier, oppermann2012generating,
   rouquier2006representation, rouquier2008dimensions,
   zheng2020thedimension,zheng2020extension,zheng2020upper} and so on.
There are a lot of triangulated categories having infinite dimension,
for instance, Oppermann and \v St'ov\'\i \v cek proved 
in \cite{oppermann2012generating} that
all proper thick subcategories of the bounded derived category
 of finitely generated modules over a Noetherian algebra
containing perfect complexes have infinite dimension.

Let $\Lambda$ be an artin algebra. Let $\mod \Lambda$ be the 
category of finitely generated right $\Lambda$-modules
and let $D^{b}(\mod \Lambda)$ be the bounded derived category of
$\mod \Lambda$. The upper bounds for the dimensions of
 the bounded derived category of
$\mod \Lambda$ can be given in terms of
the Loewy length $\LL(\Lambda)$ and the global dimension $\gldim\Lambda$ of $\Lambda$.

For a length-category $\mathcal{C}$, 
generalizing the Loewy length, Huard, Lanzilotta and Hern\'andez introduced
in \cite{huard2013layer,huard2009finitistic} the (radical) layer 
length associated with a torsion pair, which is a new measure for objects of
$\mathcal{C}$. Let $\Lambda$ be an artin algebra and $\mathcal{V}$ 
a set of some simple modules in $\mod \Lambda$.
Let $t_{\mathcal{V}}$ be the torsion radical of a 
torsion pair associated with $\mathcal{V}$ (see Section 3 for details). We use
$\ell\ell^{t_{\mathcal{V}}}(\Lambda)$ to 
denote the $t_{\mathcal{V}}$-radical layer length of $\Lambda$.
For a module $M$ in $\mod \Lambda$, we use $\pd M$ to denote the
 projective dimensions of $M$;
in particular, set $\pd M=-1$ if $M=0$.
 For a subclass $\mathcal{B}$ of $\mod \Lambda$,
  the {\bf projective dimension} $\pd\mathcal{B}$
of $\mathcal{B}$ is defined as
\begin{equation*}
\pd \mathcal{B}=
\begin{cases}
\sup\{\pd M\;|\; M\in \mathcal{B}\}, & \text{if} \;\; \mathcal{B}\neq \varnothing;\\
-1,&\text{if} \;\; \mathcal{B}=\varnothing.
\end{cases}
\end{equation*}
 Note that $\mathcal{V}$ is a finite set. So, if each simple module
in $\mathcal{V}$ has finite projective dimension,
 then $\pd \mathcal{V}$ attains its (finite) maximum.

Now, let us list some results about the upper bound of 
the dimension of bounded derived categries.
\begin{theorem} \label{thm1.1}
Let $\Lambda$ be an artin algebra and $\mathcal{V}$ 
a set of some simple modules in $\mod \Lambda$. Then we have
\begin{enumerate}
\item[(1)] {\rm (\cite[Proposition 7.37]{rouquier2008dimensions})}
 $\dim D^{b}(\mod \Lambda) \leqslant \LL(\Lambda)-1;$
\item[(2)] {\rm (\cite[Proposition 7.4]{rouquier2008dimensions} 
and \cite[Proposition 2.6]{krause2006rouquier})}
 $\dim D^{b}(\mod \Lambda) \leqslant \gldim \Lambda$;
\item[(3)] {\rm (\cite[Theorem 3.8]{zheng2020upper})}
  $\dim D^{b}(\mod \Lambda) \leqslant 
  (\pd\mathcal{V}+2)(\ell\ell^{t_{\mathcal{V}}}(\Lambda)+1)-2;$
\item[(4)] {\rm (\cite{zheng2020thedimension})}
$\dim D^{b}(\mod \Lambda) \leqslant 
2(\pd\mathcal{V}+\ell\ell^{t_{\mathcal{V}}}(\Lambda))+1.$
\end{enumerate}
\end{theorem}

For an integer $m\geqslant 0$,
we denote by $\Omega^{m}(X)$  the $m$-th syzygy of $X\in \mod \Lambda$ and
we denote by
$$\Omega^{m}(\mod \Lambda)=\{ M\;|\; M\text{ is a direct summand of }
\Omega^{m}(N)\text{ for some } N \in \mod \Lambda\}.$$
Following \cite[P. 834]{xi2017Finitistic}, $\Lambda$ is called \emph{ $m$-syzygy-finite}
if there are only finitely many non-isomorphic indecomposable modules in
$\Omega^{m}(\mod \Lambda)$. If there is some nonnegative integer $m$, such that
$\Lambda$ is $m$-syzygy-finite, then $\Lambda$ is said to be syzygy-finite.

The aim of this paper is to prove the following

\begin{theorem}{\rm (see Theorem \ref{thm-1}
   and Corollary \ref{derived-dimension1})}
   Let $A$ be an artin algebra.
  Let $\V\subseteq \S^{<\infty}$. If $\ell\ell^{t_{\V}}(A_{A})\leqslant 2$,
  then $ A$ is $(\pd \V +2)$-syzygy-finite
 and $\dim {D^{b}(\mod A)}\leqslant \pd \V+3$ and the left big finitistic dimension conjecture holds.
    \end{theorem}

    We also give examples to explain
     our results. In this case, 
     we may be able to get a better upper 
     bound on the dimension of 
   the bounded derived category of $\mod \Lambda$.
\section{Preliminaries}

\subsection{The dimension of a triangulated category}

We recall some notions from \cite{rouquier2006representation,rouquier2008dimensions,oppermann2009lower}.
Let $\mathcal{T}$ be a triangulated category and $\mathcal{I} \subseteq {\rm Ob}\mathcal{T}$.
Let $\langle \mathcal{I} \rangle$ be the full subcategory consisting of $\mathcal{T}$
of all direct summands of finite direct sums of shifts of objects in $\mathcal{I}$.
Given two subclasses $\mathcal{I}_{1}, \mathcal{I}_{2}\subseteq {\rm Ob}\mathcal{T}$, we denote $\mathcal{I}_{1}*\mathcal{I}_{2}$
by the full subcategory of all extensions between them, that is,
$$\mathcal{I}_{1}*\mathcal{I}_{2}=\{ X\mid  X_{1} \longrightarrow X \longrightarrow X_{2}\longrightarrow X_{1}[1]\;
{\rm with}\; X_{1}\in \mathcal{I}_{1}\; {\rm and}\; X_{2}\in \mathcal{I}_{2}\}.$$
Write $\mathcal{I}_{1}\diamond\mathcal{I}_{2}:=\langle\mathcal{I}_{1}*\mathcal{I}_{2} \rangle.$
Then $(\mathcal{I}_{1}\diamond\mathcal{I}_{2})\diamond\mathcal{I}_{3}=\mathcal{I}_{1}\diamond(\mathcal{I}_{2}\diamond\mathcal{I}_{3})$
for any subclasses $\mathcal{I}_{1}, \mathcal{I}_{2}$ and $\mathcal{I}_{3}$ of $\mathcal{T}$ by the octahedral axiom.
Write
\begin{align*}
\langle \mathcal{I} \rangle_{0}:=0,\;
\langle \mathcal{I} \rangle_{1}:=\langle \mathcal{I} \rangle\; {\rm and}\;
\langle \mathcal{I} \rangle_{n+1}:=\langle \mathcal{I} \rangle_{n}\diamond \langle \mathcal{I} \rangle_{1}\;{\rm for\; any \;}n\geqslant 1.
\end{align*}

\begin{definition}{\rm (\cite[Definiton 3.2]{rouquier2006representation})}\label{def2.1}
The {\bf dimension} $\dim \mathcal{T}$ of a triangulated category $\mathcal{T}$
is the minimal $d$ such that there exists an object $M\in \mathcal{T}$ with
$\mathcal{T}=\langle M \rangle_{d+1}$. If no such $M$ exists for any $d$, then we set
 $\dim \mathcal{T}=\infty.$
\end{definition}

\subsection{ Radical layer lengths and torsion pairs}

We recall some notions from \cite{huard2013layer}.
Let $\mathcal{C}$ be a {\bf length-category}, that is, $\mathcal{C}$
is an abelian, skeletally small category and every object of $\mathcal{C}$ has a finite composition series.
We use $\End_{\mathbb{Z}}(\mathcal{C})$ to denote the category of all additive functors from
$\mathcal{C}$ to $\mathcal{C}$, and use $\rad$ to denote the Jacobson radical lying in $\End_{\mathbb{Z}}(\mathcal{C})$.
For any $\alpha\in\End_{\mathbb{Z}}(\mathcal{C})$, set the {\bf $\alpha$-radical functor} $F_{\alpha}:=\rad\circ \alpha$.

\begin{definition}{\rm (\cite[Definition 3.1]{huard2013layer})\label{def2.5}
For any $\alpha, \beta \in \End_{\mathbb{Z}}(\mathcal{C})$, we define
the {\bf $(\alpha,\beta)$-layer length} $\ell\ell_{\alpha}^{\beta}:\mathcal{C} \longrightarrow \mathbb{N}\cup \{\infty\}$ via
$\ell\ell_{\alpha}^{\beta}(M)=\inf\{ i \geqslant 0\mid \alpha \circ \beta^{i}(M)=0 \}$; and
the {\bf $\alpha$-radical layer length} $\ell\ell^{\alpha}:=\ell\ell_{\alpha}^{F_{\alpha}}$.}
\end{definition}

For more information about radical layer length, we can see
\cite{huard2009finitistic,huard2013layer,zheng2020upper,zheng2020extension}.

\begin{lemma}{\rm (\cite[Lemma 2.6]{zheng2020upper})}\label{lem2.6}
Let $\alpha,\beta \in\End_{\mathbb{Z}}(\mathcal{C}) $.
For any $M\in \mathcal{C}$, if $\ell\ell_{\alpha}^{\beta}(M)=n$, then $\ell\ell_{\alpha}^{\beta}(M)=\ell\ell_{\alpha}^{\beta}(\beta^{j}(M))+j$
for any $0 \leqslant j\leqslant n$; in particular, if $\ell\ell^{\alpha}(M)=n$, then $\ell\ell^{\alpha}(F_{\alpha}^{n}(M))=0$.
\end{lemma}

Recall that a {\bf torsion pair} (or {\bf torsion theory}) for $\mathcal{C}$
is a pair of classes $(\mathcal{T},\mathcal{F})$ of objects in $\mathcal{C}$ satisfying the following conditions.
\begin{enumerate}
\item[(1)] $\Hom_{\mathcal{C}}(M,N)=0$ for any $M\in\mathcal{T}$ and $N\in\mathcal{F}$;
\item[(2)] an object $X \in \mathcal{C}$ is in $\mathcal{T}$ if $\Hom_{\mathcal{C}}(X,-)|_{\mathcal{F}}=0$;
\item[(3)] an object $Y\in\mathcal{C}$ is in $\mathcal{F}$ if $\Hom_{\mathcal{C}}(-,Y)|_{\mathcal{T}}=0$.
\end{enumerate}

For a subfunctor $\alpha$ of $1_{\C}$, we write $q_{\alpha}:=1_{\mathcal{C}}/\alpha$.
Let $(\mathcal{T},\mathcal{F})$ be a torsion pair for $\mathcal{C}$.
Recall that the {\bf torsion radical} $t$ is a functor in $\End_{\mathbb{Z}}(\mathcal{C})$ such that
$$0 \longrightarrow  t(M)\longrightarrow M \longrightarrow  q_{t}(M)\longrightarrow 0$$
is a short exact sequence and $q_{t}(M)(=M/t(M))\in \mathcal{F}$.

\subsection{Some facts}

In this section, $\Lambda$ is an artin algebra. 
Then $\mod \Lambda$ is a length-category.
For a module $M$ in $\mod\Lambda$, we use $\rad M$ 
and $\top M$ to denote the radical, socle and top of $M$ respectively.
For a subclass $\mathcal{W}$ of $\mod \Lambda$, we use $\add \mathcal{W}$ 
to denote the subcategory
of $\mod \Lambda$ consisting of direct summands of
 finite direct sums of modules in $\mathcal{W}$,
and if $\mathcal{W}=\{M\}$ for some $M\in \mod \Lambda$,
 we write $\add M:=\add \mathcal{W}$.

 Let $\S^{\infty}$ be the set of the simple modules with
 infinite projective dimension and 
 $\S^{<\infty}$ be the set of the simple 
 module with finite projective dimension. 
Let $\mathcal{S}$ be the set of the simple modules in $\mod \Lambda$, and let $\mathcal{V}$ be a subset of $\mathcal{S}$
and $\mathcal{V}'$ the set of all the others simple modules in $\mod \Lambda$, that is, $\mathcal{V}'=\mathcal{S}\backslash\mathcal{V}$.
We write $\mathfrak{F}\,(\mathcal{V}):=\{M\in\mod\Lambda\mid$ there exists a finite chain
$$0=M_0\subseteq M_1\subseteq \cdots\subseteq M_m=M$$ of submodules of $M$
such that each quotient $M_i / M_{i-1}$ is isomorphic to some module in $\mathcal{V}\}$.
By \cite[Lemma 5.7 and Proposition 5.9]{huard2013layer}, we have that
$(\mathcal{T}_{\mathcal{V}}, \mathfrak{F}(\mathcal{V}))$ is a torsion pair, where
$$\mathcal{T}_{\mathcal{V}}=\{M \in \mod \Lambda\mid\top M\in \add \mathcal{V}'\}.$$
We use $t_{\mathcal{V}}$ to denote the torsion radical of the torsion pair $(\mathcal{T}_{\mathcal{V}}, \mathfrak{F}(\mathcal{V}))$.
Then $t_{\mathcal{V}}(M)\in \mathcal{T}_{\mathcal{V}}$ and $q_{_{t_{\mathcal{V}}}}(M)\in\mathfrak{F}(\mathcal{V})$ for any
$M\in \mod \Lambda$.

\subsection{Short exact sequences and radical layer length}

\begin{lemma}\label{lemma1}
For any module $X \in \mod\Lambda$. We have

$(1)$ $ t_{\V}(\La_{\La})$ is a two side ideal 
and $t_{\V}(X)=X t_{\V}(\La_{\La})$.

$(2)$ $\rad X=X \rad(\La_{\La})$.

$(3)$ $\rad t_{\V}(\La_{\La})
=t_{\V}(\La_{\La}) \rad(\La_{\La})$.

$(4)$ $t_{\V}\rad t_{\V}(\La_{\La})=t_{\V}(\La_{\La})
 \rad(\La_{\La})t_{\V}(\La_{\La}) $.

$(5)$ $t_{\V}F_{t_{\V}}^{i}(\La_{\La})$ is 
an ideal of $\La$ for each $i \geqslant 0 $.

%$(5)$ In general, $F_{t_{\V}}^{i}(\La_{\La})=F_{t_{\V}}^{i}(\La_{\La})t_{\V}(\La_{\La}) \rad(\La_{\La})$.

 %$\rad t_{\V}\rad t_{\V}(\La_{\La})=t_{\V}(\La_{\La}) \rad(\La_{\La})t_{\V}(\La_{\La})\rad (\La_{\La}) $.
\end{lemma}
\begin{proof}
$(1)$ See \cite[Proposition 5.9(c)]{huard2013layer}.

$(2)$ See \cite[Propostion 3.5]{auslander1997representation}.

$(3)$ Let $X=t_{\V}(\La_{\La})$ and By(2).

$(4)$ By (1)(3).

$(5)$ By (1)(2)(3)(4).
\end{proof}

\begin{lemma}\label{lemma2}
For any module $X\in \mod \Lambda$, we have
$t_{\V}F_{t_{\V}}^{i}(X)=X (t_{\V}F_{t_{\V}}^{i}(\La_{\La}))$ for each $i \geqslant 0 $.
\end{lemma}
\begin{proof}
If $i=0$, by Lemma \ref{lemma1}(1).

Suppose that if $i=n$, we have $t_{\V}F_{t_{\V}}^{n}(X)=X (t_{\V}F_{t_{\V}}^{n}(\La_{\La}))$.

Now consider the case $i=n+1$.
\begin{align*}
t_{\V}F_{t_{\V}}^{n+1}(X) &=t_{\V}F_{t_{\V}}^{n}(F_{t_{\V}}(X))& \\
&=F_{t_{\V}}(X)t_{\V}F_{t_{\V}}^{n}(\La_{\La}) &\text{ (by assumption) }\\
&=\rad (t_{\V}(X))t_{\V}F_{t_{\V}}^{n}(\La_{\La})& \text{ (by }F_{t_{\V}}=\rad \circ t_{\V} ) \\
&=t_{\V}(X)\rad (\La_{\La})t_{\V}F_{t_{\V}}^{n}(\La_{\La})& \text{ (by Lemma \ref{lemma1}(2)}  )\\
&=t_{\V}(X)\rad (t_{\V}F_{t_{\V}}^{n}(\La_{\La}))&\text{ (by Lemma \ref{lemma1}(2)}  )\\
&=t_{\V}(X)F_{t_{\V}}^{n+1}(\La_{\La})& \text{ (by }F_{t_{\V}}=\rad \circ t_{\V} )\\
&=Xt_{\V }(\La_{\La}) F_{t_{\V}}^{n+1}(\La_{\La})&\text{ (by Lemma \ref{lemma1}(1)}  )\\
&=Xt_{\V } F_{t_{\V}}^{n+1}(\La_{\La})&\text{ (by Lemma \ref{lemma1}(3)(4)}  ).
\end{align*}
\end{proof}

\begin{lemma}{\rm (\cite[Lemma 3.3]{zheng2020upper})}\label{lemma3}
The functor $t_{\V}$ preserve monomorphism and epimorphism.
\end{lemma}
%\begin{proof}
%By \cite[Lemma 3.6(a)]{huard2013layer}, we know that $t_{\V}$ preserve monomorphism.
%By Lemma \ref{3.1}, we see that $\mathfrak{F}(\S)$ is closed under
%quotient modules. Moreover, by \cite[Proposition 1.3]{Beachy1971Cotorsion}, we get that
%$t_{\V}$ epimorphism.
%\end{proof}
\begin{lemma}\label{lemma4}
The functor $\rad $ preserve monomorphism and epimorphism.
\end{lemma}

\begin{proof}
Note that $\rad$ preserve monomorphism
(see \cite[Lemma 3.6(a)]{huard2013layer}) and
 epimorphism(see \cite[Chapter V, Lemma 1.1]{assem2006elements}).
\end{proof}

\begin{lemma}\label{lemma5}
For each $i \geqslant 0$,
$F_{t_{\V}}^{i}=\rad \circ t_{\V}$ and  $t_{\V}F_{t_{\V}}^{i}=\rad \circ t_{\V}$ preserve monomorphism and epimorphism.
\end{lemma}
\begin{proof}
By Lemma \ref{lemma3} and Lemma \ref{lemma4}.
\end{proof}

By Definition \ref{def2.5}, we have the following observation.
\begin{lemma}\label{lemma6}
For any module $X\in \mod \Lambda$,
we have
$t_{\V}F_{t_{\V}}^{\ell\ell^{t_{\V}}(X)}(X)=0.$
\end{lemma}
Now, we give the main theorem in this paper.
%\begin{lemma}\label{lemma7}
%For any module $Y, X\in \mod \Lambda$. If $I$ is an ideal of $\La$, and $Y$ is a submodule of $X$,
%then we have $(X/Y)I\cong (XI+Y)/Y.$
%\end{lemma}
\begin{theorem}\label{theorem1-radical-shortexact}
Let $0  \longrightarrow L\longrightarrow  M\longrightarrow N\longrightarrow 0$
be an exact sequence in $\mod \La$.
Then $$\max\{\ell\ell^{t_{\V}}(L), \ell\ell^{t_{\V}}(N)\}  \leqslant \ell\ell^{t_{\V}}(M) \leqslant \ell\ell^{t_{\V}}(L)+\ell\ell^{t_{\V}}(N).$$
In particular, if
 $\ell\ell^{t_{\V}}(L)=0,$
 then
 $  \ell\ell^{t_{\V}}(N)=\ell\ell^{t_{\V}}(M) $;
  if
 $\ell\ell^{t_{\V}}(N)=0,$
 then
 $  \ell\ell^{t_{\V}}(L)=\ell\ell^{t_{\V}}(M). $

\end{theorem}
\begin{proof}

By Lemma \ref{lemma5}, we know that $F_{t_{\V}}=\rad \circ\; t_{\V}$  preserve monomorphism and epimorphism.
Thus by \cite[Lemma 3.4(b)(c)]{huard2013layer}, we can obtain
that $\ell\ell^{t_{\V}}(L) \leqslant \ell\ell^{t_{\V}}(M) $ and
 $\ell\ell^{t_{\V}}(N) \leqslant \ell\ell^{t_{\V}}(M) $,
that is,
$$\max\{\ell\ell^{t_{\V}}(L), \ell\ell^{t_{\V}}(N)\}  \leqslant \ell\ell^{t_{\V}}(M).$$

Next, we will prove the second `$ \leqslant$'.

By Lemma \ref{lemma1}(5), we know that $t_{\V}F_{t_{\V}}^{i}(\La_{\La})$ is an ideal of $\La$ for each $i \geqslant 0 $.
By assumption, we have $M/L\cong N.$
Moreover, we get
 \begin{align*}
 ( t_{\V}F_{t_{\V}}^{\ell\ell^{t_{\V}}(N)}(M)+L)/L%&\text{ (by Lemma \ref{lemma2}})\\
  &= (M( t_{\V}F_{t_{\V}}^{\ell\ell^{t_{\V}}(N)}(\La_{\La}))+L)/L&\text{ (by Lemma \ref{lemma2}})\\
 &= (M/L)(t_{\V}F_{t_{\V}}^{\ell\ell^{t_{\V}}(N)}(\La_{\La})) &\\%\text{ (by Lemma \ref{lemma7}})\\
 &\cong N(t_{\V}F_{t_{\V}}^{\ell\ell^{t_{\V}}(N)}(\La_{\La})) & \text{ (by} M/L\cong N)\\
&=t_{\V}F_{t_{\V}}^{\ell\ell^{t_{\V}}(N)}(N)&\text{ (by Lemma \ref{lemma2}})\\
&=0&\text{ (by Lemma \ref{lemma6}}).
\end{align*}

That is, $ t_{\V} F_{t_{\V}}^{\ell\ell^{t_{\V}}(N)}(M)+L=L$. Moreover, $ t_{\V}F_{t_{\V}}^{\ell\ell^{t_{\V}}(N)}(M)\subseteq L$.
And by Lemma \ref{lemma5} and Lemma \ref{lemma6}, we have
 $$t_{\V}F_{t_{\V}}^{\ell\ell^{t_{\V}}(L)}(t_{\V}F_{t_{\V}}^{\ell\ell^{t_{\V}}(N)}(M))
\subseteq t_{\V}F_{t_{\V}}^{\ell\ell^{t_{\V}}(L)}(L)=0,$$
where we use the fact that $t_{\V}$ is idempotent, that is, $t^{2}_{\V}=t_{\V}$.
That is,
$$ t_{\V} F_{t_{\V}}^{\ell\ell^{t_{\V}}(L)+\ell\ell^{t_{\V}}(N)}(M)
=t_{\V}F_{t_{\V}}^{\ell\ell^{t_{\V}}(L)}(t_{\V}F_{t_{\V}}^{\ell\ell^{t_{\V}}(N)}(M))
\subseteq t_{\V}F_{t_{\V}}^{\ell\ell^{t_{\V}}(L)}(L)=0.$$
Thus,
$ \ell\ell^{t_{\V}}(M) \leqslant\ell\ell^{t_{\V}}(L) + \ell\ell^{t_{\V}}(N)$ by Definition \ref{def2.5}.
\end{proof}
\begin{remark}{\rm
Note that the functions Loewy length  $\LL$ and infinite layer length $\ell\ell^{\infty}$ are particular radical layer length, more details see \cite{huard2013layer,huard2009finitistic}. Corollary \ref{coro1}(1) is a classical result. The first ``$\leqslant$'' in Corollary \ref{coro1}(2)
is first established in \cite[Proposition 4.5(a)(b)]{huard2009finitistic}.
}
\end{remark}
\begin{corollary}\label{coro1}
Let $0  \longrightarrow L\longrightarrow  M\longrightarrow N\longrightarrow 0$
be an exact sequence in $\mod \La$.
Then

$(1)$ $\max\{\LL(L), \LL(N)\}  \leqslant \LL(M) \leqslant \LL(L)+\LL(N).$

$(2)$ $\max\{\ell\ell^{\infty}(L), \ell\ell^{\infty}(N)\}  \leqslant \ell\ell^{\infty}(M) \leqslant \ell\ell^{\infty}(L)+\ell\ell^{\infty}(N).$

$(3)$ if $\ell\ell^{\infty}(L)=0,$
 then
 $  \ell\ell^{\infty}(N)=\ell\ell^{\infty}(M) $;
  if
 $\ell\ell^{\infty}(N)=0,$
 then
 $  \ell\ell^{\infty}(L)=\ell\ell^{\infty}(M). $

\end{corollary}

\begin{proof}
(1)(2) are particular cases of Theorem \ref{theorem1-radical-shortexact}.

(3) if $\ell\ell^{\infty}(L)=0,$ by (2), we have
$$\ell\ell^{\infty}(N)=\max\{\ell\ell^{\infty}(L), \ell\ell^{\infty}(N)\}  \leqslant \ell\ell^{\infty}(M) \leqslant \ell\ell^{\infty}(L)+\ell\ell^{\infty}(N)=\ell\ell^{\infty}(N),$$
that is, $\ell\ell^{\infty}(N)=\ell\ell^{\infty}(N).$
Similarly,
 if
 $\ell\ell^{\infty}(N)=0,$
 then
 $  \ell\ell^{\infty}(L)=\ell\ell^{\infty}(M). $

\end{proof}

\section{Main results}

\begin{lemma}{\rm (\cite[Lemma 3.6]{huard2008An})}\label{lem-3}
Let $0  \longrightarrow X \longrightarrow Y \longrightarrow Z \longrightarrow 0$ be an exact sequence in $\mod A$.
Then we have the
following:

$(1)$ if $\pd Z$ is finite then, for any $m$ with $\pd Z\leqslant m$, there are projective $A$-modules
$P_{m}$ and $P_{m}'$
 such that
$\Omega^{m}(X)\oplus P_{m}\cong  \Omega^{m}(Y)\oplus P'_{m}$

$(2)$ if $\pd X$ is finite then, for any $m$ with $\pd X\leqslant m$, there are projective $A$-modules
$P_{m}$ and $P_{m}'$
 such that
$\Omega^{m+1}(Y)\oplus P_{m}\cong \Omega^{m+1}(Z)\oplus P'_{m}$

\end{lemma}

The following lemma is a special case of \cite[Lemma 6.3]{huard2013layer}.
\begin{lemma}{\rm (\cite[Lemma 6.3]{huard2013layer})}\label{lem-4}
Let $\V \subseteq \S^{<\infty}$ and $M\in \mod A$.
If $t_{\V}(M)\neq 0$, then $\ell\ell^{t_{\V}}(\Omega t_{\V}(M))\leqslant\ell\ell^{t_{\V}}(A_{A})-1. $
\end{lemma}
\begin{lemma}\label{lem-6}
Let $M,N \in \mod A$. If $ M\in \add (N)$,
then for any $n\geqslant 0$, we have
$\Omega^{n}(M)\in \add(\Omega^{n}(N)).$
\end{lemma}
\begin{proof}
Since $ M\in \add (N)$, we can set
$M \oplus L \cong N^{s}$ for some positive integer $n$ and $L\in \mod A$.
Thus,
$\Omega^{n}(M) \oplus \Omega^{n}(L) \cong
 \Omega^{n}(M \oplus L) \cong
  \Omega^{n}(N^{s})\cong (\Omega^{n}(N))^{s}.$
  That is, $\Omega^{n}(M)\in \add(\Omega^{n}(N)).$
\end{proof}

\begin{theorem}\label{thm-1}
Let $\V\subseteq \S^{<\infty}$. If $\ell\ell^{t_{\V}}(A_{A})\leqslant 2$,
then $\mod  A$ is $(\pd \V +2)$-syzygy-finite.
\end{theorem}
\begin{proof}
We set $\delta=\pd \V$.
If $\ell\ell^{t_{\V}}(A_{A})=0$. For any module $M\in \mod A$, we have $\ell\ell^{t_{\V}}(M)\leqslant \ell\ell^{t_{\V}}(A_{A})=0$,
that is, $\ell\ell^{t_{\V}}(M)=0$. And then $M\in \mathfrak{F}(\V)$, moreover, $\pd M\leqslant \delta .$

Now consider the case $1 \leqslant \ell\ell^{t_{\V}}(A_{A})=2$.
We have the following two canonical two short exact sequences
\begin{align}\label{exact-1}
0 \longrightarrow  t_{\V} (M) \longrightarrow M
\longrightarrow  q_{t_{\V}}(M) \longrightarrow 0,
\end{align}
\begin{align}\label{exact-2}
0 \longrightarrow t_{\V}\Omega t_{\V} (M) \longrightarrow  \Omega t_{\V} (M)
\longrightarrow  q_{t_{\V}}\Omega t_{\V} (M) \longrightarrow 0,
\end{align}
\begin{align}\label{exact-3}
0 \longrightarrow \rad t_{\V}\Omega t_{\V} (M) \longrightarrow t_{\V}\Omega t_{\V} (M)
\longrightarrow  \top t_{\V}\Omega t_{\V} (M) \longrightarrow 0.
\end{align}

 For any module $M\in \mod A$.

 If $\ell\ell^{t_{\V}}(\Omega t_{\V}(M))= 0$,
by Lemma \ref{theorem1-radical-shortexact} and sequence (\ref{exact-2}) we know that
$$\ell\ell^{t_{\V}}(t_{\V}\Omega t_{\V} (M))=\ell\ell^{t_{\V}}(\Omega t_{\V}(M))=0;$$
and
by Lemma \ref{theorem1-radical-shortexact} and sequence (\ref{exact-3}) we know that
$$ 0 \leqslant \ell\ell^{t_{\V}}(\rad t_{\V}\Omega t_{\V} (M))\leqslant \ell\ell^{t_{\V}}(\Omega t_{\V}(M))=0.$$
That is, $\ell\ell^{t_{\V}}(\rad t_{\V}\Omega t_{\V} (M))=0.$
And then $\pd \rad t_{\V}\Omega t_{\V} (M)\leqslant \delta.$

If $\ell\ell^{t_{\V}}(\Omega t_{\V}(M))= 1$. By Lemma \ref{theorem1-radical-shortexact} and sequence (\ref{exact-2})
 we have
 $\ell\ell^{t_{\V}}(t_{\V}\Omega t_{\V}(M))= \ell\ell^{t_{\V}}(\Omega t_{\V}(M))= 1$.
 By Lemma \ref{lem2.6}, we have

$$\ell\ell^{t_{\V}}(\rad t_{\V}\Omega t_{\V}(M))= \ell\ell^{t_{\V}}(t_{\V}\Omega t_{\V}(M))-1= 1-1=0.$$
Thus, $\pd \rad t_{\V}\Omega t_{\V}(M)) \leqslant \delta$.

By the short exact sequence (\ref{exact-1}) and Lemma \ref{lem-3}(1), we have
\begin{align}\label{iso-1}
\Omega^{\delta+1}t_{\V}(M)\oplus P_{1}=\Omega^{\delta}(\Omega t_{\V}(M))\oplus P_{1}
\cong \Omega^{\delta+1}(M)\oplus P_{2}.
\end{align}

By the short exact sequence (\ref{exact-2}) and Lemma \ref{lem-3}(1), we have
\begin{align}\label{iso-2}
\Omega^{\delta+1}(t_{\V}\Omega t_{\V}(M))\oplus P_{3}\cong
\Omega^{\delta+1}(\Omega t_{\V}(M))\oplus P_{4}=\Omega^{\delta+2}t_{\V}(M)\oplus P_{4}.
\end{align}

By the short exact sequence (\ref{exact-3}) and Lemma \ref{lem-3}(2), we have
\begin{align}\label{iso-3}
\Omega^{\delta+1}(t_{\V}\Omega t_{\V}(M))\oplus P_{5}\cong
\Omega^{\delta+1}\top t_{\V}\Omega t_{\V}(M)\oplus P_{6}.
\end{align}

And then we have the following isomorphisms
\begin{align*}
\Omega^{\delta+2}(M)\oplus P_{4} \oplus P_{5}\cong &\Omega(\Omega^{\delta+1}(M)\oplus P_{2})\oplus P_{4} \oplus P_{5}\\
\cong &\Omega(\Omega^{\delta+1}t_{\V}(M)\oplus P_{1})\oplus P_{4} \oplus P_{5}\;\;(\text{ by (\ref{iso-1})})\\
\cong &\Omega^{\delta+2}(t_{\V}(M))\oplus P_{4} \oplus P_{5}\\
\cong &(\Omega^{\delta+1}t_{\V}\Omega t_{\V}(M)\oplus P_{3})\oplus P_{5}\;\;(\text{ by (\ref{iso-2})})\\
\cong &(\Omega^{\delta+1}t_{\V}\Omega t_{\V}(M)\oplus P_{5})\oplus P_{3}\;\;\\
\cong &(\Omega^{\delta+1}\top t_{\V}\Omega t_{\V}(M)\oplus P_{6}) \oplus P_{3}\;\;(\text{ by (\ref{iso-3})})\\
\in& \add (\Omega^{\delta+2}(A/\rad A)\oplus A).\;\;(\text{ by Lemma \ref{lem-6}})
\end{align*}
By assumptions and Lemma \ref{lem-4}, we always have $\ell\ell^{t_{\V}}(\Omega t_{\V}(M))\leqslant 1$.
Thus, for any module $M\in \mod A$, we have
$\Omega^{\delta+2}(M) \in \add (\Omega^{\delta+2}(A/\rad A)\oplus A)$.
That is, $\mod A$ is $(\delta+2$)-syzygy-finite.
\end{proof}
\begin{corollary}\label{cor-1}
If $\ell\ell^{\infty}(A_{A})\leqslant 2$,
then $ A$ is syzygy-finite.
\end{corollary}
\begin{proof}
Let $\V=\S^{<\infty}$, we have
$\ell\ell^{t_{\V}}(A_{A})=\ell\ell^{\infty}(A_{A})\leqslant 2$ by \cite[Example 5.8(1)]{huard2013layer}.
And then by \ref{thm-1}, we know that $A_{A}$ is syzygy-finite.
\end{proof}
The notion of the left big finitistic dimension conjecture can be seen in \cite{rickard2019unbounded}.
\begin{corollary}\label{cor-3}
Let $A$ be a finite dimensional algebra over a field $K$.
Let $\V\subseteq \S^{<\infty}$. If $\ell\ell^{t_{\V}}(A_{A})\leqslant 2$,
then $$\lFindim A <\infty,$$
where $\lFindim A =\sup \{ \pd M \;|\; M \text{ is a left }
 \Lambda \text{-module with } \pd M <\infty \}$; that is,
 the left big finitistic dimension conjecture holds.
\end{corollary}
\begin{proof}
By\cite[Definition 4.1, Definition 4.2, Corollary 7.3, Theorem 4.3]{rickard2019unbounded} and Theorem \ref{thm-1}.
\end{proof}

As a consequence we have the following upper bound on the dimension
$\dim D^{b}(\mod A)$ of the
bounded derived category of $\mod A$
in the sense of Rouquier(see
\cite{rouquier2008dimensions,rouquier2006representation,krause2006rouquier}).
Here, we have an interesting corollary as follows
\begin{corollary}\label{derived-dimension1}
Let $\V\subseteq \S^{<\infty}$.
Suppose that $\ell\ell^{t_{\V}}(A_{A})\leqslant 2$.
Then $\dim {D^{b}(\mod A)}\leqslant \pd \V+3$.
\end{corollary}
\begin{proof}
By Theorem \ref{thm-1} and \cite[Corollary 3.6]{Asadollahi2012On}.
\end{proof}

\begin{corollary}{\rm
  Let $A$ be an artin algebra. 
  Let $\V\subseteq \S^{<\infty}$.
   If $\ell\ell^{t_{\V}}(A_{A})\leqslant 2$,
  then $$\Psi\dim(\mod A) <\infty,$$
  where $\Psi\dim(\mod A) $ is defined in
   \cite{Lanzilotta2017Igusa}.}
\end{corollary}
\begin{proof}
  By Corollary \ref{cor-1} and \cite[Thoerem 3.2]{Lanzilotta2017Igusa}.
\end{proof}
\begin{example}
{\rm (\cite{zheng2020upper})
Consider the bound quiver algebra $\Lambda=kQ/I$, where $k$ is an algebraically closed field and $Q$
is given by
$$\xymatrix{
&1 \ar@(l,u)^{\alpha_{1}}\ar[r]^{\alpha_{2}}  \ar[ld]_{\alpha_{m+1}}\ar[rd]^{\alpha_{m+2}}
&2\ar[r]^{\alpha_{3}}&{3}\ar[r]^{\alpha_{4}}  &{4}\ar[r]^{\alpha_{5}}&\cdots\ar[r]^{\alpha_{m}}&m\\
m+1&&m+2&&&&
}$$
and $I$ is generated by
$\{\alpha_{1}^{2},\alpha_{1}\alpha_{m+1},\alpha_{1}\alpha_{m+2},\alpha_{1}\alpha_{2},
\alpha_{2}\alpha_{3}\cdots\alpha_{m}\}$ with $m\geq 10$.
Then the indecomposable projective $\Lambda$-modules are
$$\xymatrix@-1.0pc@C=0.1pt
{ &  &  &1\edge[lld]\edge[ld]\edge[d]\edge[dr]
&  && 2\edge[d] &&&& & &&&&  &&&&  &\\
&1& m+1&m+2&2\edge[d] && 3\edge[d] && 3\edge[d] &&&& &&&& &\\
P(1)=&  &  &    &3\edge[d] &P(2)=&4\edge[d]  &P(3)=&4\edge[d] &P(m+1)=m+1,&P(m+2)=m+2&\\
&  &  &  &\vdots\edge[d]&&\vdots\edge[d]&&\vdots\edge[d]&& &&&& &\\
&  &  &  & \;m-1, &&\;m, && \;m,  &&&& &&&\\
&  &  &  &  &&&&  & &&&& & &&&& &&&\\
}$$
and $P(i+1)=\rad P(i)$ for any $2 \leqslant i\leqslant m-1$.

We have
\begin{equation*}
\pd S(i)=
\begin{cases}
\infty, &\text{if}\;\;i=1;\\
1,&\text{if} \;\;2 \leqslant  i\leqslant m-1;\\
0,&\text{if}\;\; m \leqslant  i\leqslant m+2.
\end{cases}
\end{equation*}
So $\mathcal{S}^{\infty}=\{S(1)\}$ and $\mathcal{S}^{<\infty}=\{ S(i)\mid 2\leqslant i\leqslant m+2\}$.

Let $\mathcal{V}:=\{S(i)\mid 3\leqslant i \leqslant m-1\}\subseteq\mathcal{S}^{<\infty}$.
 Then $\pd\V =1$ and $\ell\ell^{t_{\mathcal{V}}}(\Lambda)=2$(see \cite[Example 4.1]{zheng2020upper})

(1) By Theorem \ref{thm1.1}(1), we have $\dim D^{b}(\mod \Lambda) \leqslant \LL(\Lambda)-1=m-2.$

(2) By Theorem \ref{thm1.1}(3), we have
  $\dim D^{b}(\mod \Lambda) \leqslant (\pd\mathcal{V}+2)(\ell\ell^{t_{\mathcal{V}}}(\Lambda)+1)-2=7.$

(3) By Theorem \ref{thm1.1}(4), we have
  $\dim D^{b}(\mod \Lambda) \leqslant 2(\pd\mathcal{V}+\ell\ell^{t_{\mathcal{V}}}(\Lambda))+1=7.$

(4) By Corollary \ref{derived-dimension1},
  $\dim D^{b}(\mod \Lambda) \leqslant \pd\mathcal{V}+3=4.$
That is, we can get a better upper bound.
}
\end{example}

\begin{example}
{\rm (\cite[Example 3.21]{zheng2020extension})
Consider the bound quiver algebra $\Lambda=kQ/I$, where $k$ is a field and $Q$ is given by
$$\xymatrix{
&2n+1\\
{2n}&1\ar[l]^{\alpha_{2n}}\ar[u]^{\alpha_{2n+1}}\ar[r]^{\alpha_{1}} \ar[d]^{\alpha_{n+1}}
&2\ar[r]^{\alpha_{2}} &{3}\ar[r]^{\alpha_{3}}&\cdots \ar[r]^{\alpha_{n-1}}&{n}\\
&n+1\ar[r]^{\alpha_{n+2}}
&n+2\ar[r]^{\alpha_{n+3}}&n+3\ar[r]^{\alpha_{n+4}}&\cdots\ar[r]^{\alpha_{2n-1}}&2n-1
}$$
and $I$ is generated by
$\{\alpha_{i}\alpha_{i+1}\;| \;n+1\leqslant i\leqslant 2n-1\}$ with $n\geqslant 6$.
%\edge[lld]
Then the indecomposable projective $\La$-modules are
$$\xymatrix@-1.0pc@C=0.1pt
{& &1\edge[d]\edge[ld]\edge[rd]\edge[rrd]&&&&   & 2\edge[d]  &&&&  & &&&&&   &&& &\\
&n+1  &2\edge[d]  &2n&2n+1&&  &3\edge[d]  &&&& & 3\edge[d] &&&& &j\edge[d] &&&  & &\\
P(1)=& &3\edge[d]   &&&&P(2)=&4\edge[d]  &&&&P(3)=&4\edge[d] &&&&P(j)=&j+1,&&&&P(l)=l,&\\
& &\vdots\edge[d]&&&&   &\vdots\edge[d]&&&& &\vdots\edge[d]&&&&  & &&& &\\
&&n,&&&&  &n,   &&&& &n,  &&&& &  &&&  &\\
&&  &&&&  &   &&&& &  &&&&  &  &&&  &\\
}$$
where $n+1 \leqslant j\leqslant 2n-2$, $2n-1 \leqslant l\leqslant 2n+1$ and $P(i+1)=\rad P(i)$
for any $2 \leqslant i\leqslant n-1$.

We have
\begin{equation*}
\pd S(i)=
\begin{cases}
n-1, &\text{if}\;\;i=1;\\
1,&\text{if} \;\;2 \leqslant  i\leqslant n-1;\\
0,&\text{if} \;\; i=n, 2n, 2n+1;\\
2n-1-i,&\text{if}\;\; n+1 \leqslant  i\leqslant 2n-1.
\end{cases}
\end{equation*}
So $\S^{<\infty}=\{$all simple modules in $\mod \Lambda\}$.
Let $\V:=\{S(i)\mid 2\leqslant i \leqslant n\}(\subseteq\mathcal{S}^{<\infty})$.
 Then $\pd\V =1$ and $\ell\ell^{t_{\mathcal{V}}}(\Lambda)=2$(see \cite[Example 3.21]{zheng2020extension})

(1) By Theorem \ref{thm1.1}(1), we have $\dim D^{b}(\mod \Lambda) \leqslant \LL(\Lambda)-1=n-1.$

(2) By Theorem \ref{thm1.1}(2), we have
  $\dim D^{b}(\mod \Lambda) \leqslant \gldim \Lambda =n-1.$

(3) By Theorem \ref{thm1.1}(3), we have
  $\dim D^{b}(\mod \Lambda) \leqslant (\pd\mathcal{V}+2)(\ell\ell^{t_{\mathcal{V}}}(\Lambda)+1)-2=7.$

(4) By Theorem \ref{thm1.1}(4), we have
  $\dim D^{b}(\mod \Lambda) \leqslant 2(\pd\mathcal{V}+\ell\ell^{t_{\mathcal{V}}}(\Lambda))+1=7.$

(5) By Corollary \ref{derived-dimension1},
  $\dim D^{b}(\mod \Lambda) \leqslant \pd\mathcal{V}+3=4.$

That is, we also can get a better upper bound than \cite[Example 4.1]{zheng2020upper}.

}
\end{example}

\vspace{0.6cm}

{\bf Acknowledgements.}
This work was supported by the National Natural Science Foundation of China(Grant No. 12001508).

%\bibliographystyle{amsalpha}%
%\bibliographystyle{alpha}%
%\bibliographystyle{amsplain}%
%\bibliographystyle{plain}%
%��д��ĸ˳��, X. Ma
%\bibliographystyle{siam}%  ��д��ĸ˳��, X. Ma,  ���±�����б��
%\bibliographystyle{unsrt}%
%\bibliographystyle{abbrv}
%\bibliography{ref}

%\bibliography{ieeetr}

\end{document}